\documentclass{amsart}
\usepackage{amsmath,amsthm,amsfonts,amssymb,amscd,mathrsfs}
\usepackage{psfrag,graphicx}

\usepackage{epic,eepic}
\usepackage{color}



\theoremstyle{plain}
\newtheorem{main}{Theorem}

\newtheorem{Thm}{Theorem}[section]
\newtheorem{Lem}[Thm]{Lemma}
\newtheorem{Prop}[Thm]{Proposition}

\theoremstyle{remark}
\newtheorem{Def}[Thm] {Definition}
\newtheorem{Rem}[Thm] {Remark}
\newtheorem{Con}{Conjecture}

\newtheorem{Que}[Thm] {Question}

\long\def\begcom#1\endcom{}

\newcommand{\Leb}{\operatorname{Leb}}

\newcommand{\orb}{\operatorname{orb}}
\newcommand{\length}{\operatorname{\length}}

\newcommand{\Diff}{\operatorname{Diff}}

\newcommand{\cl}{\operatorname{cl}}

\def\Diff{\operatorname{Diff}}

\def\supp{\operatorname{supp}}

\def\id{\operatorname{id}}
\def\length{\operatorname{length}}

\begin{document}

\title[Entropy and periodic orbits ]
      {Entropy and periodic orbits for  equivalent smooth flows}

\author[Liao, Sun]
{Gang Liao$^{\dag}$, Wenxiang Sun$^{\ddag}$}

\thanks{$^{\ddag}$ Sun is supported by National Natural Science Foundation of China  \# 10831003 and
Education Ministry of China}

\thanks{$^{\dag, \ddag}$
School of Mathematical Sciences, Peking University, Beijing 100871,
China} \email{liaogang@math.pku.edu.cn}
\email{sunwx@math.pku.edu.cn}

\date{May, 2011}

\maketitle

\begin{abstract}
Given any $K>0$, we construct two  equivalent $C^2$ flows, one of
which has positive topological entropy larger than $K$ and admits
zero as the exponential growth of periodic orbits, in contrast,  the
other has zero topological entropy and super-exponential growth of
periodic orbits. Moreover we establish a $C^{\infty}$ flow on
$\mathbb{S}^2$ with super-exponential growth of periodic orbits,
which is also equivalent to another flow with zero exponential
growth of periodic orbits. On the other hand, any two dimensional
flow has only zero topological entropy.
\end{abstract}


\section{Introduction}

 It is a major goal in the theory of dynamical systems to determine
the mechanisms which create deterministic chaos. One key ingredient
causing chaotic behavior is the positivity  of entropy, and the
larger the entropy is, the more complicated the dynamics. It is of
interesting to calculate  the entropy for particular maps under
study, but any of the standard definitions of entropy makes this a
difficult task. A classical way for proving that a map  $f$ has
positive topological entropy is to show that the number of periodic
orbits of period $ n$ for $f$ grows exponentially fast when
$n\rightarrow \infty$. This motivation  came from many  analysis and
summary  on dynamical structures.  In his prize essay \cite{Po1}, H.
Poincar{\'e} was the first to imagine around 1890 the existence of
transverse homoclinic intersections, that later was proved to be the
limit of infinitely many periodic points by G. D. Birkhoff
\cite{Birkhoff}. In 1965, S. Smale introduced a general geometrical
model: Horseshoe contains all the complicated phenomena discovered
by Poincar{\'e} and Birkhoff, and can also be described by symbolic
coding.  Indeed, homoclinic intersections give birth to very rich
dynamics: positive topological entropy  and infinity periodic
points. Actually, the coexistence of the positive entropy and the
positive exponential growth of periodic points has been established
for open dense systems  \cite{ Pujals-Sambarino, Bonatti-Gan-Wen,
Crovisier}. So, in most situations, ``positive topological entropy"
is synonymous of `` many periodic orbits ".

Let $M$ be a compact Riemannian manifold without boundary. Denote by
$\Diff^r(M)$ the set of $C^r$ diffeomorphisms of $M$ and,
 by $\mathcal{X}^r(M)$ the set of $C^r$ vector fields on $M$, both endowed with the $C^r$ topology, respectively.

For $f\in \Diff^r(M)$, denote  the set of isolated periodic points
of period $n$ (i.e. the isolated fixed points of $f^n$ ) by
$$P_n(f) =\{ \,\,\mbox{isolated}\,\, x\in M \,\mid f^n(x)=x\, \},$$
and define the exponential growth rate of periodic points by
$$EP(f)
=\limsup_{n\rightarrow +\infty} \frac{1}{n}\log \,\sharp\,
P_{n}(f),$$ where $\sharp A$ is the cardinal number of a set $A$. In
1978,  Bowen \cite{Bowen1} asked the following
 question:
\begin{Que}Is the property that
$$EP = h$$
generic with respect to the $C^r$ topology?
\end{Que}
For Axiom A systems \cite{Bowen3, Bowen, Bowen2} one in fact has
$$EP = h.$$   Beyond uniform
hyperbolicity, Katok \cite{Katok} stated that, if $f$ is a
$C^{1+\alpha}$ diffeomorphism of a compact surface $S$ with positive
topological entropy, then $EP(f) \geq h(f)$. For any hyperbolic
ergodic measure $\mu$, the authors and Tian \cite{Liao-Sun-Tian}
established the equality between metric entropy and the exponential
growth rate of those periodic measures approximating $\mu$.  Exactly
in broad situations, due to the absence of uniform hyperbolicity,
the periodic orbits can grow much faster than entropy. Linking with
a conjecture of Palis \cite{Palis},  we mention two well known
obstructions for the hyperbolicity: homoclinic tangencies
\cite{Newhouse} and heterodimensional cycles \cite{Abraham}. In
\cite{Kaloshin1} Kaloshin showed super-exponential growth of
periodic orbits for a residual subset in some $C^r$-domain
($r\geq2$) with persistent homoclinic tangencies. In \cite{BDF}
Bonatti, D{\'\i}az and Fisher proved super-exponential growth of
periodic points for homoclinic classes with persistent
heterodimensional cycles.

In the content of density,   in 1965 Artin and  Mazur \cite{AM}
proved that: there exists a dense set $\mathcal{D}$ of $C^r$ maps
such that for any map $f\in \mathcal{D}$,  the number $\sharp
P_n(f)$ grows at most exponentially with $n$ (see also
\cite{Kaloshin2} for an extension concerning hyperbolic periodic
points). For a vector field $X$, use  $\phi_X$ to write the flow
induced by $X$. Set
$$P_t(\phi_{X}) =\{\,\,\mbox{isolated}\,\orb(\phi_{X}, x)\mid \phi_{X}(x,0)=\phi_{X}(x,s)\mbox{ for some }0\leq s \leq t \}$$ and
define the exponential growth rate of periodic orbits by
$$EP(\phi_{X})
=\limsup_{t\rightarrow +\infty} \frac{1}{t}\log \,\sharp\,
P_{t}(\phi_{X}).$$  Artin and Mazur \cite{AM} asked the following
question  for vector fields:
\begin{Que}Does the property that $EP(\phi_{X})<\infty$ hold for a
dense subset of $\mathcal{X}^r(M)$?
\end{Que}

As we know, this  question is far from being resolved, because the
approaches concerning diffeomorphisms don't apply directly to flows.
To continue the story of periodic orbits for dense systems we are in
a position to understand more on the growth of periodic orbits of
flows.

 Two flows
$\phi, \,\psi$ defined on a smooth manifold $M$ are equivalent if
there exists a homeomorphism $\pi$ of $M$ that sends each orbit of
$\phi$ onto an orbit of $\psi$ while preserving the time orientation
:
$$\{\phi(x,t) \mid \,\,\,t\in \mathbb{R}\}=\{\pi^{-1}\psi(\pi(x),t)\mid\,\, t\in \mathbb{R}\},\quad \forall\,x\in M.$$
 Going back to the study of  Lorenz
attractors \cite{Guckenhei-Willians, Willians} the kneading
sequences were introduced to be invariants for equivalence. In
general cases, it is not easy to find quantities preserved by
equivalence.  Topological entropy of a flow $\phi$ indicates, as
usual, that for its time one map $\phi_1$, that is,
$h(\phi)=h(\phi_1)$. Topological entropy is an invariant for
equivalent homeomorphisms (Theorem 7.2 in \cite{Walters}), while
finite non-zero topological entropy for a flow cannot be an
invariant because its value is affected by time reparameterization.
For equivalent flows without fixed points the extreme value  0 and
infinite  entropy are invariant, while the sign of finite non-zero
entropy  are preserved (see \cite{Ohno}, \cite{SunVar},
\cite{Thomas1}, \cite{Thomas2}). In equivalent flows with fixed
points there exists a counterexample, constructed by Ohno
\cite{Ohno}, showing that neither 0 nor $\infty$ topological entropy
is preserved by equivalence. The two flows constructed in
\cite{Ohno} are suspensions of a transitive subshift and thus are
not differentiable. Note that a differentiable flow on a compact
manifold cannot have $\infty$ entropy (see Theorem 7.15 in
\cite{Walters}).  Ohno \cite{Ohno} in 1980 asked  the following:
\begin{Que}Is 0 topological entropy an invariant for
equivalent differentiable flows? \end{Que}

In \cite{SunYoungZhou},  Sun, Young and Zhou constructed two
equivalent $C^\infty$  flows with a singularity, one of which has
positive topological entropy while the other has zero topological
entropy. This gives a negative answer to Ohno's question.

Likewise as entropy,  $EP=0$ or $EP=\infty$ is invariant for
equivalent homeomorphisms and also for equivalent flows without
fixed points, see \cite{Ohno, SunVar}. For topological flows with
fixed points, neither extreme growth rate, $EP=0$ nor $EP=\infty$ is
preserved for equivalence \cite{ SunZhang}. Moreover, there exists a
pair of equivalent topological flows with fixed points such that one
of which has $\infty$ topological entropy and $0$ growth rate of
periodic orbits but the other has $0$ topological entropy and
$\infty$ growth rate of periodic orbits \cite{ Sun-Zhang-Zhou}.

In the present paper we are going to study in the world  of
smoothness and  consider the following question:

\begin{Que}Is 0 or $\infty$ value of $EP$ invariant for
equivalent differentiable flows? \end{Que}

 There are fruitful
dynamical properties varying in the
 differentiability, for instance, symbolic extension which  is exactly a suitable
 candidate to ``measure" the dependence of
entropy structure on the smoothness of underlying systems.  Here we
call a system $(M,f)$ has a symbolic extension if there is a
subshift $(Y,g)$ over finite alphabets and a continuous surjection
$\pi: Y\rightarrow M$ such that $f\circ\pi=\pi\circ g$.
\begin{eqnarray*}Y&\stackrel{g}{\longrightarrow}&Y\\
 \,\,\,\pi\big{\downarrow}&&\big{\downarrow}\pi\\
M&\stackrel{f}{\longrightarrow}&M
 \end{eqnarray*}
A symbolic extension $(Y,g,\pi)$
 for which $h_{\nu}(g)=h_{\mu}(f)$ for every $g$-invariant measure $\nu$ with $\pi_{*}\nu=\mu$ is
 viewed as a good model and is called a principal symbolic
 extension.
In the context of $C^{\infty}$, Newhouse \cite{New89} showed upper
semi-continuity of metric entropy and  Buzzi \cite{Buzzi} further
established asymptotical entropy expansiveness which, together with
a criterion of Boyle, D. Fiebig and U. Fiebig \cite{BFF} : if $f$ is
asymptotically entropy expansive, then $f$ has   a principal
symbolic extension, implies all $C^{\infty}$ maps admit a principal
symbolic extension. In \cite{DN} Downarowicz and Newhouse
constructed a $G_{\delta}$ set of $C^1$ area-preserving
diffeomorphisms in which everyone has no symbolic extension.  They
described the entropy structure of $C^{2}$ differentiability  by
following:
\begin{Con}Every $C^2$ map has a symbolic extension.  \end{Con}
This conjecture has been proven by Downarowicz-Maass \cite{DM} for
interval maps and by Burguet \cite{Burguet2} for surface maps.
Altogether, we have the following intuitions to reveal various differentiability:\\
 \begin{eqnarray*}
C^1 \,\,\,\mbox{differentiability} &\longleftrightarrow&
\mbox{generically no symbolic extension}\\[2mm]
C^2  \,\,\,\mbox{differentiability} &\longleftrightarrow&
\mbox{ symbolic extension}\quad\mbox{( generically not principal)}\\[2mm]
C^{\infty} \,\,\,\mbox{differentiability} &\longleftrightarrow&
\mbox{ principal symbolic extension}.
 \end{eqnarray*}

   We call a flow $\phi$ has  a ( principal ) symbolic extension if its time one map $\phi(1,\cdot)$ has  a ( principal ) symbolic
   extension.
    As we have stated,  every
$C^{\infty}$ system has a principal symbolic extension.  It is well
known that symbolic systems with finite alphabets have finite $EP$.
Our first theorem says that the finiteness of $EP$ can't be
inherited from its principal symbolic extension although the under
system agrees the same entropy with the upper symbolic extension.
This means that  entropy is not enough to exhaust the difference of
complexity even if in the category of $C^{\infty}$.  Furthermore, we
are going to show that the extreme growth rate of periodic orbits
can't be preserved for orbit equivalent  $C^{\infty}$ flows.

\begin{main}\label{Main theorem1} There exist two
 equivalent $C^{\infty}$ flows $\varphi$ and $\widehat{\varphi}$ on
the sphere $\mathbb{S}^2$ satisfying:
$$EP(\varphi)=0,\,\,\,EP(\widehat{\varphi})=\infty.$$
\end{main}
\begin{Rem} Recall that L. S. Young
\cite{LSYoung} has proven zero entropy  for all
 surface flows. However,  much different from entropy  Theorem
 \ref{Main theorem1}
 exhibits
 the existence  of two dimensional flows with super-exponential growth of periodic orbits.\end{Rem}
 \begin{Rem}From our construction the statement of Theorem \ref{Main theorem1} can hold for any manifold of dimension $\geq 2$ if a suitable embedding is taken.
 Here we only emphasis on the existence and thus omit details for general discussions.\end{Rem}

Next we return to the relationship of $EP$ and $h$ for equivalent
flows. As mentioned at beginning for almost ( open dense  ) systems
positive entropy enjoys the company of positive $EP$. In the next
theorem we are close to be
 tightrope walkers since our examples are excluded by those open
 dense sets.
 We start by supposing
that $f : M \rightarrow M$ is a $C^{\infty}$ diffeomorphism of a
smooth compact Riemannian manifold $M$ with $\dim M = m \geq 4$ with
the following properties: (1) $f$ has positive topological entropy
and (2) $f$ is minimal in the sense that all forward orbits are
dense in $M$ (or equivalently closed invariant sets are either empty
or the entire space). An example of such an $f$ was constructed by
Herman \cite{Herman}. Using the constant function $I : M \rightarrow
\mathbb{R},\, I(x) = 1,$ one gets a suspension manifold $\Omega$ and
a smooth vector field $X$ associated with this flow. Since $M$ is of
dimension $\geq4$, we know that $\dim \Omega\geq5$.

\begin{main}\label{Main theorem2} Given $K>0$, there exist two
$C^{2}$ equivalent flows $\psi$ and $\widehat{\psi}$ on $\Omega$
satisfying the following:\smallskip

$(1)$ $h(\psi)>K$ and $EP(\psi)=0$;\smallskip

$ (2)$  $h(\widehat{\psi})=0$ and $EP(\widehat{\psi})=\infty$.
\end{main}

\section{Two dimensional equivalent flows:  proof of Theorem \ref{Main theorem1} }

Throughout this section, we use $D$ to denote the  two-dimensional
unit disk
$$\{x=(x_1,x_2)\in \mathbb{R}^2\mid x_1^2+x_2^2\leq1\}.$$
In order to obtain Theorem \ref{Main theorem1}, we proceed the main
stategery  of the proof as follows: firstly we arrange many enough
periodic orbits with the same period on $D$ and then we change their
periods on different orbits according to the applications of
different equivalent flows. Precisely, take a strictly decreasing
sequence $\{a_i\}$ with $a_i=\frac1i$. For $i\geq 2$, let
$$l_i=\min\{a_i-a_{i+1}, a_{i-1}-a_i\},$$
$$b_{i,j}=a_i+\frac{jl_i}{2^{2^{i+2}}},\,\,\mbox{where}\,\,-2^{2^{i}}\leq j\leq 2^{2^{i}}.$$
Define the strips centered at the circle $r=a_i$ by
$$L_i=\{x\mid\,a_i-\frac{l_i}{4}\leq x\leq a_i+\frac{l_i}{4}\}.$$
Denote $$I_0=0,\quad I_i=\sum_{j=1}^i 2^{2^{j}+1}+i,\quad i\geq 1.$$
Rearrange the sequence $\{b_{i,j}\}\cup\{1\}$ by decreasing order,
denoting
\begin{eqnarray*}
b_1=1,\quad b_i=b_{s,j} \quad \mbox{for}\quad i-1=I_{s-1}+j, \quad
|j|\leq 2^{2^{s}}.
\end{eqnarray*}

For the purpose to  get periodic orbits supporting on $r=b_{i}$, we
give a $C^{\infty}$ function $\alpha_0$ on $[0,1]$:
\begin{equation*}
\alpha_0(x)=\begin{cases}e^{\frac{1}{(x-b_i^2)(x-b_{i+1}^2)}}
\,\,&\mbox{for}\,\,b_i^2<x<b_{i+1}^2,\,i\geq 1,\\
0&\mbox{for}\,\,\,x=0\,\,\mbox{or}\,\,b_{i}^2,\,i\geq 1.\end{cases}
\end{equation*}
Consider a standard differential equation
\begin{equation*}
\begin{cases}
\frac{dx}{dt}&=-y+\alpha_0(x^2+y^2)x,\\
\frac{dy}{dt}&=x+\alpha_0(x^2+y^2)y.
\end{cases}
\end{equation*}
Let $x=r\cos\theta,y=r\sin\theta$, then
\begin{equation*}
\begin{cases}
r\frac{dr}{dt}&=\alpha_0(r^2)r^4,\\
\frac{d\theta}{dt}&=1.
\end{cases}
\end{equation*}

For the sake of writing,  denote the vector field
$Z_0(x,y)=(-y+\alpha_0(x^2+y^2)x,\,x+\alpha_0(x^2+y^2)y)$. We can
see that $\phi_{Z_0}$ has periodic orbits $r=b_i$ with the period
$2\pi$. To get different exponential growth rate of periodic orbits,
next we will change the period for different $i$.

{\it (1) Constriction of a flow on $D$ with $EP=\infty$.}

Take a $C^{\infty}$ function $\beta_1: \mathbb{R}\rightarrow
\mathbb{R}$ such that
$$\beta_1(x)=a_i^2\quad \mbox{for}\quad x\in L_i.$$
Consider the vector field $Z_1=\beta_1(r^2)Z_0$, then the number of
$2\pi n^2$-periodic orbits is $2^{2^{n}+1}+1$. Thus,
$$EP(\phi_{Z_1})
=\limsup_{t\rightarrow +\infty} \frac{1}{t}\log \,\sharp\,
P_{t}(\phi_{Z_1})\geq\limsup_{n\rightarrow +\infty} \frac{1}{2\pi
n}\log \,(2^{2^{n}+1}+1)=\infty.$$

{\it (2) Construction of a flow on $D$ with $EP=0$.}

Take a $C^{\infty}$ function $\beta_2: \mathbb{R}\rightarrow
\mathbb{R}$ such that
$$\beta_2(x)=2^{-2^i},\,\,\,\mbox{for}\,\,x\in L_i.$$
Consider the vector field $Z_2=\beta_2(r^2)Z_0$, then the number of
$2\pi 2^{2^i}$-periodic orbits is $2^{2^{i}+1}+1$. Thus,
\begin{eqnarray*}EP(\phi_{Z_2}) &=&\limsup_{t\rightarrow +\infty} \frac{1}{t}\log
\,\sharp\, P_{t}(\phi_{Z_2})\\&=&\limsup_{n\rightarrow +\infty}
\frac{1}{2\pi 2^{2^n}}\log
\,(\sum_{i=1}^n2^{2^{i}+1}+1)\\
&\leq& \limsup_{n\rightarrow +\infty} \frac{1}{2\pi 2^{2^n}}\log
\,n(2^{2^{n}+1}+1)\\
&=&0.\end{eqnarray*}Finally  one can see  that the two smooth flows
$\phi_{Z_1}$ and $\phi_{Z_2}$  from our constructions are in fact
equivalent.

\noindent {\bf Proof \,of\,  Theorem \ref{Main theorem1}\,\,\,\,\,}
Let $\mathbb{S}^{2+}=\{(x_1,x_2,x_3)\in \mathbb{S}^2 \mid x_3\geq0
\}$, $\mathbb{S}^{2-}=\{(x_1,x_2,x_3)\in \mathbb{S}^2 \mid x_3\leq0
\}$. Define projection $\varrho(x_1,x_2,x_3)=(x_1,x_2)$ and
$\varrho_+=\varrho\mid _{\mathbb{S}^{2+}}$, $\varrho_-=\varrho\mid
_{\mathbb{S}^{2-}}$. Next we  embed the flows $\phi_{Z_i}$ of $D$
into $\mathbb{S}^2$ by the double cover $\varrho$.

\begin{figure}[h]\label{iterate}
\begin{center}
\begin{picture}(180,120)(-30,-12)

\put(-50,0){\begin{picture}(100,100)

\put(-10,20){\circle{20}}
 \put(-10,20){\circle{40}}
 \put(-10,20){\circle{60}}

\put(-10,10){\vector(1,0){}}
  \put(-10,30){\vector(-1,0){}}

  \put(-10,0){\vector(1,0){}}
  \put(-10,40){\vector(-1,0){}}

  \put(-10,-10){\vector(1,0){}}
  \put(-10,50){\vector(-1,0){}}

\put(-10,20){\circle*{2}}

\put(-10,-20){\makebox(0,0){$D$}}

\end{picture}
}

\put(120,0){\begin{picture}(100,100) \put(-10,20){\circle{60}}
\qbezier[50](-40,20)(-40,30)(-10,30)
\qbezier(-40,20)(-40,10)(-10,10) \qbezier[50](-10,30)(20,30)(20,20)
\qbezier(20,20)(20,10)(-10,10)

\qbezier[50](-31.2,41.2)(-31.2,45)(-10,45)
\qbezier[50](-10,45)(11.2,45)(11.2,41.2)
\qbezier(11.2,41.2)(11.2,37.4)(-10,37.4)
\qbezier(-10,37.4)(-31.2,37.4)(-31.2,41.2)

\qbezier(-10,-5)(-31.2,-5)(-31.2,-1.2)
\qbezier[50](-31.2,-1.2)(-31.2,2.6)(-10,2.6)
\qbezier[50](-10,2.6)(11.2,2.6)(11.2,-1.2)
\qbezier(-10,-5)(11.2,-5)(11.2,-1.2)

\put(-11,45){\vector(-1,0){}}

\put(-11,30){\vector(-1,0){}}
 \put(-9,10){\vector(1,0){}}

 \put(-9,37.4){\vector(1,0){}}

 \put(-11,2.6){\vector(-1,0){}}

\put(-9,-5){\vector(1,0){}}

\put(-10,-20){\makebox(0,0){$\mathbb{S}^2$}}

\end{picture}
}

\end{picture}
\end{center}
\caption{\,Flows on $D$ and $\mathbb{S}^2$}
\end{figure}
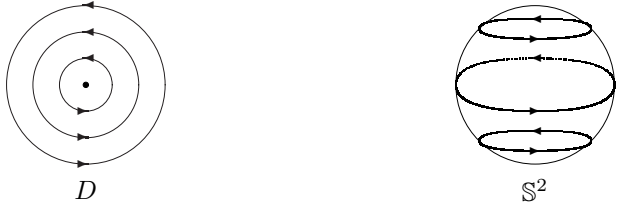


 Precisely, let \begin{equation*}\varphi=\begin{cases}
 \varrho_-^{-1}\circ \phi_{Z_2}\circ \varrho_-\,\,\, &\mbox{for}\,\,x\in \mathbb{S}^{2-},\\
\varrho_+^{-1}\circ \phi_{Z_2}\circ \varrho_+\,\,\,
&\mbox{for}\,\,x\in \mathbb{S}^{2+}.
 \end{cases}\end{equation*}
\begin{equation*}\widehat{\varphi}=\begin{cases}
 \varrho_-^{-1}\circ \phi_{Z_1}\circ \varrho_-\,\,\, &\mbox{for}\,\,x\in \mathbb{S}^{2-},\\
\varrho_+^{-1}\circ \phi_{Z_1}\circ \varrho_+\,\,\,
&\mbox{for}\,\,x\in \mathbb{S}^{2+}.
 \end{cases}\end{equation*}

Then $\varphi$ and $\widehat{\varphi}$ are equivalent flows on
$\mathbb{S}^2$ with the desired property:
$$EP(\varphi)=0,\,\,\,EP(\widehat{\varphi})=\infty.$$
\hfill$\Box$

\section{High dimensional equivalent flows: proof of Theorem \ref{Main theorem2} }

\subsection{Basic notions and technique lemmas on suspension flows }
$ $\\

Before the construction, we need do some preliminaries.

Let $(M, \mathcal{B}(M),\mu)$ be a probability space and  $f$ be a
$\mu$-measurable  map.

Consider the space $\Omega = M \times [0, 1]/ \sim$, where $\sim$ is
the identification of $(y, 1)$ with $(f(y), 0)$. The standard
suspension of $f$ is the flow $\phi$ on $\Omega$ defined by $\phi
(y, s) = (y, t + s)$, for $0 \leq t +s < 1$. A standard argument as
in \cite{STLiao} shows that $ \Omega$ is a $C^{\infty}$ smooth
compact Riemannian manifold  and $\phi$ is $C^{\infty}$ provided
$f:M\rightarrow M$ is a $C^{\infty}$ diffeomorphism on $C^{\infty}$
smooth manifold $M$.  If $f : M \rightarrow M$ is minimal as a
homeomorphism, then $\phi$ is a minimal flow.
\begin{Prop}[Proposition 2.15 of \cite{SunYoungZhou}]\label{enttropy of suspend}
Let $\mu$ be an invariant ergodic measure of $f$ on $M$. We define
$$\int_{\Omega}g d\overline{\mu}:=\int_{E}\int_0^{1}g(x,t) dtd\mu, \,\,\forall g\in C^0(\Omega).$$
Then we have
$$h_{\overline{\mu}}(\psi)=h_{\mu}(f).$$
\end{Prop}

\begin{Def}
Suppose $\phi$ is a measurable flow on a Borel probability space
$(M, \mathcal{B}(M), \nu)$ and $\Omega$ is divided into disjoint
invariant measurable sets $F_1$ and $F_2$ such that $\mu(F_1) = 1$
and $\mu(F_2)=0$. Further suppose that $\theta(t, x)$ is a real
measurable function defined on
$(-\infty,+\infty)\times(\Omega\setminus F_2) = \mathbb{R}\times
F_1$ with the following properties for every fixed
$x\in F_1:$\\

(1) $\theta(t, x)$ is continuous and non-decreasing in $t$;

(2) $\theta(t + s, x) = \theta(s, x) + \theta(t, \phi_s(x))$ for all
$t$ and $s$;

(3) $\theta(0, x) = 0$, $\lim_{t\rightarrow+\infty}\theta(t, x)
=\infty$, $\lim_{t\rightarrow-\infty}\theta(t, x) =\infty$.\\

\noindent {Then} $\theta$ is called an additive function of $\Omega$
with carrier $F_2$. An additive function is said to be integrable if
it is integrable in $\Omega$ for every fixed $t$.
\end{Def}

For  a non-negative, integrable function $a(x)$, we define $$
E_{\mu}(a) = \int_{E}a(x) d\mu(x) .$$

\begin{Lem}\label{subadditive}If $\phi$ is a measurable flow on a Borel probability space
$(\Omega, \mathcal{B}(\Omega), \nu)$ and $a(x)$ is a non-negative,
integrable function satisfying  $$E_{\mu}(a) = \int_{\widetilde{E}}
a(x) d\mu(x) > 0,$$ then the function
$$\theta(t, x) = \int_{0}^t a(\phi_s(x)) ds$$ is an integrable additive function.
\end{Lem}
For a proof see Theorem 3.1 in \cite{Totoki}.

 \begin{Def} The
function $\theta(t, x)$ in Lemma \ref{subadditive} is called the
additive function defined by $a(x)$.
\end{Def}

\begin{Lem}\label{suspend of measures}
Let $\mu$  be an invariant probability measure of $f$ on $E$. Assume
$\theta(t, x)$ is the additive function defined by $a(x)$ with
$0<E_{\mu}(a)<\infty$. We define
$$\int_{\Omega}g d\widehat{\mu}:=\frac{1}{E_{\mu}(a)}\int_{E}\int_0^{\theta(x,t)}g(x,t) dtd\mu, \,\,\forall g\in C^0(\Omega).$$
Then $\widehat{\mu}$ is an invariant measure of $\phi_t$ on
$\Omega$. Further, $\widehat{\mu}$ is ergodic if $\mu$ is ergodic.
\end{Lem}
The proof is elementary and omitted.

\begin{Lem}[Lemma 2.4 of \cite{SunYoungZhou}]\label{large measure}Suppose $(M,f)$ is a minimal
homeomorphism. Then for any $\varepsilon>0$, there exists
$L(\varepsilon)>0$ such that for any ergodic invariant measure
$\mu$, we have
$$\mu(B_{M}(x,\varepsilon))\geq \frac{1}{L(\varepsilon)}>0,\,\forall x\in M.$$

\end{Lem}

\begin{Lem}[Corollary 2.12 of \cite{SunYoungZhou}]\label{atomic measures}Assume that $\phi_t$ is the standard suspension of a minimal homeomorphism $(M,f)$ from above, $X$ is
the vector field that induces $\phi_t$ and $\alpha\in C^1(M, [0,
1])$. Denote by $\phi_{\alpha X}$  the flow induced by the vector
field $\alpha X$ on $\Omega$. For any $x\in M$, define $\gamma(x)$
by: \begin{equation*} \begin{cases}\phi_{\alpha X}((x,0), \gamma(x))
= \phi_{\alpha X}((x, 0),1) = (f(x), 0), \,\,\,&(x, 0) \neq
f^{-1}(p) \,\,\mbox{and}\,\, (x, 0) \neq p\, ;\\
\gamma(x) = +\infty,\,\, & (x, 0) = f^{-1}(p)\,\, \mbox{or}\,\, (x,
0) = p\,.\end{cases}\end{equation*} If $E_{\mu}(\gamma)=+\infty$ for
 any non-atomic ergodic measure $\mu$ of $f$,  then $\phi_{\alpha X}$ has only atomic invariant Borel
probability measures.
\end{Lem}

$$$$
\subsection{Proof of Theorem \ref{Main theorem2}}~\newline\\
 {\bf Step 1}\,\,\, Construction of a flow with zero entropy.\\

 Take $p_0=[(x_0,0)]=\pi(x_0,0)$ where $\pi$ is
the quotient map $\pi: M\times \mathbb{R}\rightarrow \Omega$.

Without loss of generality, we can assume the existence of a
coordinate chart $(\widetilde{V} , \xi)$ of $\Omega$ satisfying the
following:

$(i)$ There exists an open set $V$ of $\Omega$Ħ, such that $p_0\in
V$ and $\overline{V} \subset\widetilde{ V}$ .

$ (ii)$ $\xi(p_0) = 0,  \xi(V ) = B^{m+1}(0,1), \xi(\widetilde{V}) =
B^{m+1}(0,2)$, where $2\leq m = \dim M$.

$(iii)$ There exists $ i_1\in \mathbb{N}$  such that
$$\cl(\pi(B_M(x_0, i_1^{-1} )\times {\{0\}}))\subset V$$
 and
$$\xi(\pi(B_M(x_0, i_1^{-1} )\times \{0\})) \subset \mathcal{R} =\{x = (x_1,\cdots , x_m, x_{m+1}) :
x_{m+1} = 0\},$$ where $\cl(F)$ denotes the closure of a subset
$F\subset \Omega$.

$ (iv)$ $\exists\,i_2\in \mathbb{N}$ such that
$$B_{\Omega}(p_0, i_2^{-1} )\subset V$$
 and $$\xi(B_{\Omega}(p_0, i^{-1})) = B^{m+1}(0,i^{-1})$$
  for any $i_2 <i\in \mathbb{N}$.
   Under these assumptions, there exists $i_3\in \mathbb{N}$ and $i_2 < i_3$,
with the property that for any $i\geq 0$ there exists $1\gg
l_{i_3+i}> 0$ such that $$\cl(\,\pi( B_M(x_0, \frac{1}{ i_3 + i}
)\times [-l_{i_3+i},0]))\subset B_{\Omega}(p,\frac{1}{i_2+i}).$$

We set $i_0 := \max\{i_1, i_2, i_3\}$. For any $i>i_0$, by Lemma
\ref{large measure}, there exists $L( \frac1i)$ such that for any
ergodic measure $\tau$ of $f$, we have $$ \tau(B_M(f^{-1}x_0,
\frac1i ))\geq \frac{1}{L(\frac1i)}:=\delta(i)>0 .$$ We define
$\beta_{-1} := 1$ and $\beta_{i-1}
:=\frac{l_{i_0}+i}{i_0+i}\delta(i_0+i)$ for  $i\geq 1.$ We need the
following Lemma to construct satisfactory smooth vector fields.

\begin{Lem}\label{smooth number} For a given sequence of positive numbers $1 =\beta_0 >\beta_1 >\beta_2 >\cdots
>\beta_i>\cdots$, there exists a $C^{\infty}$ function $w : B^{m+1}(0,2) \rightarrow [0,
1]$ such that

$(1)$ $w=0$  if and only if $x = 0$;

$(2)$ $\|w\mid_{B^{m+1}(\frac{1}{i+1})}\|_{C^0}\leq \beta _{i-1},
\,\,i=0,1,2,\cdots;$

$(3)$ $w\mid _{B^{m+1}(2)\setminus B^{m+1}(1)}=1$.

\end{Lem}

\begin{proof}
Without loss of generality, we assume that
$\lim_{i\rightarrow\infty}\beta_i= 0$.  Let $\Psi(t)$ be the
function:
\begin{equation*}\Psi(t)= \begin{cases}e^{-\frac1t},\quad &0 < t \leq
1,
\\  0, \quad &-1 < t \leq 0.
\end{cases}\end{equation*}
 Let $\{\beta_i\}$ be as
above and suppose ${c_i}$ is any decreasing sequence of positive
numbers $1 =c_{-1} >c_0 >c_1 >\cdots>c_i >\cdots$, with
$\lim_{i\rightarrow\infty}c_i = 0$. For $t <c_0$, let $g(t)$ be the
function on $[-1, c_0]$ defined by the series: $$g(t)
=\sum_{i=1}^{\infty} 2^{-i-1}\beta_{i-1}\Psi(t - c_i).$$ This series
is monotone increasing in $i$ and converges uniformly. It is zero on
$[-1, 0]$ and positive on $(0, c_0]$. For any $k$ and $0 < t < c_k$
we have
$$\sum_{i=k+1}^{\infty} 2^{-i-1}\beta_{i-1}\Psi(t -
c_i)<\frac{\beta_k\Psi(1)}{2^k} .$$ Further, since the derivatives
of the partial sums of this series converge uniformly, we may take
the derivative of the sum and we have that:
$$g'(t)=\sum_{i=1}^{\infty} 2^{-i-1}\beta_{i-1}\Psi'(t - c_i)<\infty.$$
By induction, we may also conclude that $$g^{(l)}(t)
=\sum_{i=1}^{\infty} 2^{-i-1}\beta_{i-1}\Psi^{(l)}(t - c_i)$$
converges uniformly and
$$\lim_{t\rightarrow0}g^{(l)}(t) = 0.$$ We can now clearly extend $g(t)$ to the
interval $[0, 2]$ in such a way that $g$ is $C^{\infty}$ and $g(t) =
1$ for $t\in [1, 2]$. To finish the proof of the lemma, we set $c_i
=\frac{1}{i+1}$ in the above construction of $g(t)$ and let $w(x) =
g(|x|)$ on $B^{m+2}(2)$. Because of the construction  this function
is $C^{\infty}$ smooth at 0 and on $B^{m+2}(2)$.
\end{proof}

Using Lemma \ref{smooth number}, one can find a $C^{\infty}$
function $\omega_1 : \xi(\widetilde{V})\rightarrow [0, 1]$ with the
properties:

$(i)$ $\omega_1\mid_{\xi(\widetilde{V}\setminus V)}\equiv 1$;

$(ii)$ $\| \omega_1 \mid_{B^{m+1}(0,\frac{ 1}{ i_2+i}
)}\|\leq\beta_{i-1}$;

$(iii)$  $\omega_1(0) = 0$ and $0 <\omega_1(a)\leq 1$ for $ 0\neq
a\in \xi(\widetilde{V})$.

 We then define a function $\alpha  \in C^{\infty}(\Omega, [0,
1])$ as follows:
\begin{equation*}\alpha(q)
:=\begin{cases}\omega_1\circ \xi(q),\,\,q\in
\widetilde{V};\\1,\,\,q\in \Omega\setminus
\widetilde{V}.\end{cases}\end{equation*}

Then
$$\|\alpha\mid_{B_{\Omega}(p,\frac{1}{i_2+i})}\|_{C^0}=\sup_{x\in
B_{\Omega}(p,\frac{1}{i_2+i})}\{\alpha(x)\}\leq \beta_{i-1},$$
 where we
assume, without loss of generality, that$\|\xi\|_{C^0}\leq 1$.   We
then define $Y :=\alpha X$ and let $\phi_t$ denote the flow induced
by $Y$. Recall the function $\gamma : M \rightarrow R\cup
\{\infty\}$ in Lemma \ref{smooth number} and observe that for any
$x\in B_{M}(f^{-1}(x_0),\frac{1}{i_2+i})$,
$$l_{i_0+i}=\int_{t(x)}^{\gamma(x)}\sqrt{<\alpha(\phi_s(x)X(\phi_s(x)), \alpha(\phi_s(x)X(\phi_s(x))))>}ds,$$
where $t(x)>0$ satisfies $\phi_{t(x)}(x)=\phi_{1-l_{i_0+i}}(x)$.
Then
$$\gamma(x)\geq \gamma(x)-t(x)\geq
\frac{l_{i_0+i}}{\|\alpha\mid_{B_{\Omega}(p,\frac{1}{i_2+i})}\|\|X\|}\geq
\frac{l_{i_0+i}}{\beta_{i-1}\|X\|}=\frac{i_0+i}{\delta(i_0+i)\|X\|}$$
for any $x\in B_M(f^{-1}(x_0),\frac{1}{i_0+i})$. Thus,
$$\gamma\mid_{B_M(f^{-1}(x_0),\frac{1}{i_0+i})}\geq \frac{i_0+i}{\delta(i_0+i)\|X\|}.$$

For any ergodic measure $\nu$ of $f$,
$$E(\gamma)=\int_M\gamma(x)d\nu(x)
\geq\frac{i_0+i}{\delta(i_0+i)\|X\|}\nu(B_M(f^{-1}(x_0),\frac{1}{i_0+i}))\geq
\frac{i_0+i}{\|X\|}\rightarrow +\infty,$$ as $i\rightarrow +\infty$.
So, by Lemma \ref{atomic measures} all ergodic measures of
$\phi_{\alpha X}$ are atomic, which implies $h(\phi_{\alpha X})=0$.
In fact, $\phi_{\alpha X}=\phi_t$ has only one invariant measure
$\delta_{p_0}$.

Take another point $p_1\in \Omega\setminus \cl(\widetilde{V})$.
Choose a suitable smooth  coordinate chart $(\widetilde{U},\zeta)$
of $\Omega$ satisfying that

$(i)$ there is an open set $U\subset \cl(U)\subset
\widetilde{U}\subset \cl(\widetilde{U})\subset \Omega\setminus
\cl(\widetilde{V})$.

$(ii)$  $\zeta(p_1)=0,$ $\zeta(U)=B^{m+1}(0,4)$,
$\zeta(\widetilde{U})=B^{m+1}(0,8)$.

 $(iii)$ the flow $\psi_0:=\zeta\circ\phi_{X}\circ\zeta^{-1}\mid_{\zeta(\widetilde{U})}$
induced by the following standard differential equation
\begin{equation*}\label{ simple flow}
\begin{cases}
\dot{x}_1=1;\\
\dot{x}_2=0;\\
\,\,\,\vdots\\
\dot{x}_{m+1}=0.
\end{cases}
\end{equation*}

We select a $C^2$ function $\eta: \mathbb{R}^{m+1}\rightarrow
\mathbb{R}$ satisfying

$(1)$
\begin{equation*}\eta(x)=\begin{cases}(\sum_{i=2}^{m+1}x_i^2)^{\frac32}\,\,\,
\,\,\,&x\in
B^{m+1}(0,2),\,-1\leq x_1\leq1,\\
(\sum_{i=2}^{m+1}x_i^2)^{\frac32}+e^{\frac{1}{x_1^2-1}}\,\,\,&x\in
B^{m+1}(0,2),\,x_1>1\,\mbox{or}\,x_1<-1;\end{cases}\end{equation*}

$(2)$ $\frac{1}{20}< \eta(x)<20$, when $2\leq\|x\|\leq 4$.

$(3)$ $\eta(x)=1$, when $\|x\|\geq 4$.

Noticing that $m+1\geq 5$, it is easy to verify that
$\frac{1}{\eta(x)}$ is integrable with respect to Lebesgue measure
on $B^{m+1}(0,8)$. Define a  function $\alpha_1: \Omega\rightarrow
\mathbb{R}$ as follows:
\begin{equation*}\alpha_1(q)=\begin{cases}\omega_1\circ \xi^{-1}(q)\,\,\,&q\in \widetilde{V};\\
\eta\circ \zeta^{-1}(q)\,\,\,&q\in \widetilde{U};\\
 1\,\,\, &q\in
\Omega\setminus (\widetilde{U}\cup \widetilde{V}).
\end{cases}\end{equation*}
Then $\alpha_1X$ induces a flow $\phi_{\alpha_1X}$ on $\Omega$.
Denote $$F_0=\{x\in B^{m+1}(0,4)\mid -1\leq x_1\leq 1,\quad
x_2=\cdots=x_{m+1}=0\}.$$ Then all points contained in $F_0$ are
singularities of $\phi_{\alpha_1X}$. In what follows, we will show
that the integration of $\frac{1}{\eta(x)}$
  can  guarantee the modifications in $\widetilde{U}$ don't
  contribute entropy of the consequent flows.

\begin{Prop}\label{zero entropy}
$h(\phi_{\alpha_1X})=0$.
\end{Prop}
\begin{proof}Given a flow $\phi$ on $\Omega$, for any Borel set $B\in
\mathcal{B}(\Omega)$, $p\in \Omega$ and $t>0$, define
$$I(t,p, \phi, B)=\{0\leq s\leq t\mid \phi(p,s)\in B\}.$$
and
$$J(t,p, \phi, B)=\Leb(I(t,p, \phi, B))=\int_{0}^{t}\chi_{B}(\phi(p,s))ds,$$
where \begin{equation*} \chi_{B}(x)=\begin{cases}1\,\,&x\in
B,\\0\,\,&x\in \Omega\setminus B.\end{cases}
\end{equation*}
Arbitrarily taking  an open set $U_0\in \Omega\setminus
(\widetilde{U}\cup \widetilde{V})$, since $\phi_{\alpha X}$ has only
one invariant probability measure $\delta_{p_0}$,  for any $p\in
\Omega\setminus (\widetilde{U}\cup \widetilde{V})$ we have
$$\lim_{t\rightarrow +\infty}\frac{J(t,p, \phi_{\alpha X}, U_0)}{t}=0.$$
By definition of $\eta$ one can obtain $\|\eta\|\leq20$, which
implies that
\begin{eqnarray*}J(t,p,\phi_{\alpha_1X},\widetilde{U})&=&\int_{I(t,p,
\phi_{\alpha X}, \widetilde{U})}\frac{1}{\alpha_1(\phi_{\alpha
X}(p,s))}ds\\&\geq&\int_{I(t,p, \phi_{\alpha X},
\widetilde{U})}\frac{1}{20}ds\\&=& \frac{1}{20}J(t,p, \phi_{\alpha
X}, \widetilde{U}).\end{eqnarray*} Define
$\lambda(t)=J(t,p,\phi_{\alpha_1X},\widetilde{U})+J(t,p,
\phi_{\alpha X}, \Omega\setminus \widetilde{U})$. Noting that
$\alpha_1(q)=\alpha(q)$ for $q\in \Omega\setminus \widetilde{U}$, we
have
$$J(\lambda(t), p, \phi_{\alpha_1 X},
U_0)=J(t,p, \phi_{\alpha X},  U_0).$$ Therefore,
\begin{eqnarray*}\frac{1}{\lambda(t)}\int_0^{\lambda(t)}\chi_{U_0}(\phi_{\alpha_1 X}(p,s))ds &=&\frac{J(t,
p, \phi_{\alpha X},
U_0)}{J(t,p,\phi_{\alpha_1X},\widetilde{U})+J(t,p, \phi_{\alpha X},
\Omega\setminus
\widetilde{U})}\\
&\leq&\frac{J(t, p, \phi_{\alpha X}, U_0)}{\frac{1}{20}J(t,p,
\phi_{\alpha X}, \widetilde{U})+J(t,p, \phi_{\alpha X},
\Omega\setminus \widetilde{U})}.
\end{eqnarray*}
Moreover,
\begin{eqnarray*}&&\lim_{t\rightarrow+\infty}\frac{J(t, p,
\phi_{\alpha X}, U_0)}{t}=\lim_{t\rightarrow+\infty}\frac{J(t, p,
\phi_{\alpha X}, \widetilde{U})}{t}=0,\\&&
\lim_{t\rightarrow+\infty}\frac{J(t,p, \phi_{\alpha X},
\Omega\setminus \widetilde{U})}{t}=1.\end{eqnarray*} We deduce that
\begin{eqnarray*}&&\limsup_{t\rightarrow+\infty}\frac{1}{\lambda(t)}\int_0^{\lambda(t)}\chi_{U_0}(\phi_{\alpha_1X}(p,s))ds\\
&\leq& \limsup_{t\rightarrow+\infty}\frac{J(t, p,
\phi_{\alpha X}, U_0)}{\frac{1}{20}J(t,p, \phi_{\alpha X},
\widetilde{U})+J(t,p, \phi_{\alpha X},
\Omega\setminus \widetilde{U})}\\
&=&\limsup_{t\rightarrow+\infty}\frac{\frac1tJ(t, p, \phi_{\alpha
X}, U_0)}{\frac{1}{20t}J(t,p, \phi_{\alpha X},
\widetilde{U})+\frac1tJ(t,p, \phi_{\alpha X}, \Omega\setminus
\widetilde{U})}\\&=&0.\end{eqnarray*} So, for any ergodic invariant
measure $\nu$ of $\phi_{\alpha_1X}$, $\supp(\nu)\cap U_0=\emptyset$.
All ergodic measures are  supported  on $\widetilde{U}\cup
\widetilde{V}$. Besides, all ergodic measures in  $\widetilde{U}$
are atomic.
  Thus  $h(\phi_{\alpha_1X})=0$.
\end{proof}

$$$$
{\bf Step 2}\,\,\,Construction of a flow with positive entropy.\\

We select a smooth function $\widehat{\omega}_1:
\mathbb{R}^{m+1}\rightarrow \mathbb{R}$ satisfying

$(1)$ $\widehat{\omega}_1(x)=\|x\|^2=\sum_{i=1}^{m+1}x_{i}^2$, when
$\|x\|\leq \frac12$;

$(2)$ $\frac14<\widehat{\omega}_1(x)<2$, when $\frac12<\|x\|<1$;

$(3)$ $\widehat{\omega}_1(x)=1$, when $\|x\|\geq 1$.

We can see that $\frac{1}{\widehat{\omega}_1(x)}$ is integrable with
respect to Lebesgue measure on $B^{m+1}(0,2)$ since $m+1\geq 3$.

 Define a  new smooth function
$\widehat{\alpha}_1: \Omega\rightarrow \mathbb{R}$ as follows:
\begin{equation*}\widehat{\alpha}_1(p)=\begin{cases}\widehat{\omega}_1\circ \xi^{-1}(p)\,\,\,&p\in \widetilde{V};\\
\eta\circ \zeta^{-1}(p)\,\,\,&p\in \widetilde{U};\\
 1\,\,\, &p\in
\Omega\setminus (\widetilde{U}\cup \widetilde{V}).
\end{cases}\end{equation*} Then
$$a_1:=\int_{\Omega}\frac{1}{\widehat{\alpha}_1(p)}d\overline{\mu}(p)<\infty.$$
As appoint before, denote by  $\phi_{\widehat{\alpha}_1 X}$ the flow
induced by $\widehat{\alpha}_1 X$. Recalling Herman's example, $f$
preserves an ergodic measure $\mu$ equivalent to the Riemannian
volume. In Herman's construction, for any $N\in \mathbb{N}$, $f^N$
also preserves the  volume measure $\mu$ and is minimal. So for any
$K_1>0$ we can take $N$ large so that
$h_{\mu}(f^N)=Nh_{\mu}(f)>a_1K_1$ (Here we don't change the
construction of  the manifold $\Omega$). Without loss of generality,
in what follows we assume $h_{\mu}(f)>a_1K_1$.
\begin{Prop}\label{positive entropy}$h(\phi_{\widehat{\alpha}_1
X})>K_1$.\end{Prop}
\begin{proof}First Lemma \ref{suspend of measures}
and Lemma \ref{enttropy of suspend} apply to give that
$h_{\bar{\mu}}(\phi_{ X})=h_{\mu}(f)>a_1K_1$. Define a new measure
$\widehat{\mu}_1$ on $\Omega$ as follows
$$\widehat{\mu}_1(B)=\int_{B}d\widehat{\mu}_1(x)=\int_{B}\frac{1}{\widehat{\alpha}_1(x)}d\bar{\mu}(x)$$
for all $B\in \mathcal{B}(\Omega)$. By Lemma \ref{suspend of
measures}, $\widehat{\mu}_1$ is an invariant ergodic measure of
$\phi_{\widehat{\alpha}_1 X}$ and
$$\widehat{\mu}_1(\Omega)\leq \int_{\Omega}\frac{1}{\widehat{\alpha}_1(x)}d\bar{\mu}(x)<\infty.$$
Noting that $\mu$ is ergodic, by Theorem \ref{enttropy of suspend}
we have
$$h_{\widehat{\mu}_1}(\phi_{\widehat{\alpha}_1 X})\widehat{\mu}_1(\Omega)=h_{\bar{\mu}}(\phi_{X})\bar{\mu}(\Omega)$$
which gives rise to $h_{\widehat{\mu}_1}(\phi_{\widehat{\alpha}_1
X})\geq \frac{h_{\bar{\mu}}(\phi_{X})}{a_1}>K_1$. By the variational
principle (see for example $\S$8.2 of \cite{Walters}) we conclude
that
$$h(\phi_{\widehat{\alpha}_1 X})\geq
h_{\widehat{\mu}_1}(\phi_{\widehat{\alpha}_1 X})>K_1.$$
\end{proof}
\begin{Prop}\label{equivalence}$\phi_{\alpha_1X}$ and
$\phi_{\widehat{\alpha}_1X}$ are
equivalent.\end{Prop}\begin{proof}The identity map on $\Omega$ takes
orbits of one flow $\phi_{\alpha_1X}$ to orbits of the other
$\phi_{\widehat{\alpha}_1X}$ since the singular points $p_0, p_1$
mapped to themselves and elsewhere $\alpha_1$ and
$\widehat{\alpha}_1$ are positive.   The assumption that $\alpha_1$
and $\widehat{\alpha}_1$ are non-negative also implies preservation
of time orientation and hence the equivalence of $\phi_{\alpha_1X}$
and $\phi_{\widehat{\alpha}_1X}$.\end{proof}

 {\bf Step 3}\,\,\, Tear  the segment $A_0$ to be ball-like.\\

 We begin by choosing a smooth function $\gamma_0: \mathbb{R}\rightarrow
\mathbb{R}$ as follows \begin{eqnarray*} \gamma_0(x)=\begin{cases}0
\,\,\,&x\leq -1;\\e^{\frac{1}{x^2-1}}\,\,\,&-1<x<1;
\\0\,\,\,&x\geq1.
\end{cases}
\end{eqnarray*}
Clearly $\gamma_0(x)$ is $C^{\infty}$ smooth and has zero
derivatives of all orders at $-1$ and $1$.  Then consider two
families of $C^{\infty}$ curves:
\begin{eqnarray*}\rho_a(s)&=&(s, a\gamma_0(s)) \,\,\,\,\,0\leq a\leq
1;\\[2mm]
\sigma_b(s)&=&\begin{cases}(s, \gamma_0(s)+b) \,\,&0\leq b\leq
1;\\
(s,(2-b)(\gamma_0(s)+1)+2(b-1))\,\,&1\leq b\leq 2,\end{cases}
\end{eqnarray*}where $-2<s<2$. Direct computations give rise to the
following expressions:
\begin{eqnarray*}\frac{d\rho_a(x_1)}{dx_1}&=&(1, a\gamma_0'(x_1)) \,\,\,\,\,0\leq
a<
1,\\
\frac{d\sigma_b(x_1)}{dx_1}&=&\begin{cases}(1, \gamma_0'(x_1))
\,\,&0\leq b<
1;\\
(1,\,(2-b)\gamma_0'(x_1))\,\,&1\leq b\leq2.\end{cases}
\end{eqnarray*}
Denote  regions
$$U_1=\{(x_1, x_2)\mid x_2=\rho_a(x_1),
\,\,\,-2<x_1<2,\,0\leq a\leq 1\},$$
$$V_1=\{(x_1, x_2)\mid x_2=\sigma_b(x_1),
\,\,\,-2<x_1<2,\,0\leq b\leq 1\},$$
$$W_1=\{(x_1, x_2)\mid x_2=\sigma_b(x_1),
\,\,\,-2<x_1<2,\,1\leq b\leq 2\}.$$ Now we define a projection
$\triangle: U_1\cup V_1\cup W_1\rightarrow U_1\cup V_1\cup W_1$ as
follows
\begin{eqnarray*}
\triangle(x)=(x_1,0)\quad &\mbox{if}\,\,x=(x_1,x_2)\in \rho_a;\\
\triangle(x)=(x_1,b) \quad &\mbox{if}\,\,x=(x_1,x_2)\in \sigma_b.\\
\end{eqnarray*}
Obviously $\Delta$ is continuous. Moreover,
$$\triangle\mid_{\rho_a}: \rho_a\rightarrow (-2,2)\times \{0\},\,\,\triangle\mid_{\sigma_b}: \sigma_b\rightarrow (-2,2)\times \{b\} $$
are both onto $C^{\infty}$ diffeomorphisms.
  For every $1\leq k\leq m$, consider the standard
embedding $\varrho_k: \mathbb{R}^k\rightarrow \mathbb{R}^{m+1}$
given by
$$\varrho_k(x_1,x_2,\cdots,x_k)=(x_1,x_2,\cdots,x_k,0,\cdots,0).$$ The
projection $\pi_k$ is defined by
$$\pi_k(x_1,x_2,\cdots,x_k,x_{k+1},\cdots,x_{m+1})=x_k.$$ Let
$\varphi_0(y,t):=\zeta\circ \varphi_{\alpha_1 X}(\zeta^{-1}(y), t)$
for $y\in B^{m+1}(0,8)$. Denote a  restricted flow $\widehat{\phi}$
on $U_1\cup V_1\cup W_1$
$$\widehat{\phi}((x_1,x_2),t): =(\pi_1\circ\varphi_0(\varrho_2((x_1,x_2)),t), \pi_2\circ\varphi_0(\varrho_2((x_1,x_2)),t)).$$
Now we define a flow $\varphi_1$ on $U_1\cup V_1\cup W_1$ as follows
$$\varphi_1\mid_{\rho_a}=(\triangle\mid_{\rho_a})^{-1}\circ\widehat{\phi}\circ (\triangle\mid_{\rho_a}).$$
$$\varphi_1\mid_{\sigma_b}=(\triangle\mid_{\sigma_b})^{-1}\circ\widehat{\phi}\circ (\triangle\mid_{\sigma_b}).$$
Next, we calculate the expression of the vector field $Z(x)$
associated with the flow $\varphi_1$.

Case 1:  $x=(x_1,x_2)\in\rho_a$.  $a\gamma_0(x_1)=x_2$, so
$a=\frac{x_2}{\gamma_0(x_1)}$. It follows that
\begin{eqnarray*}\frac{d\varphi_1(x,t)}{dt}\mid_{t=0}&=&d((\triangle\mid_{\rho_a})^{-1})\frac{d}{dt}\widehat{\phi}\circ
(\triangle\mid_{\rho_a})\mid_{t=0}\\
&=&d((\triangle\mid_{\rho_a})^{-1})\frac{d}{dt}\widehat{\phi}((x_1,0),t)\mid_{t=0}\\
&=&(\frac{d}{dt}\widehat{\phi}((x_1,0),t)\mid_{t=0},\,\,
u(a)\gamma_0'(x_1)\frac{d}{dt}\widehat{\phi}((x_1,0),t)\mid_{t=0})\\
&=&\eta(\varrho_1(x_1))(1,\,\,a\gamma_0'(x_1));\end{eqnarray*}

 Case
2:  $x=(x_1,x_2)\in \sigma_{b}$ for $0\leq b<1$.
$x_2=\gamma_0(x_1)+b$, so $b=x_2-\gamma_0(x_1)$.
\begin{eqnarray*}\frac{d\varphi_1(x,t)}{dt}\mid_{t=0}&=&d((\triangle\mid_{\sigma_b})^{-1})\frac{d}{dt}\widehat{\phi}\circ
(\triangle\mid_{\sigma_b})\mid_{t=0}\\
&=&d((\triangle\mid_{\sigma_b})^{-1})\frac{d}{dt}\widehat{\phi}((x_1,b),t)\mid_{t=0}\\
&=&\eta(\varrho_2(x_1,x_2-\gamma_0(x_1)))(1,\,\gamma_0'(x_1));\end{eqnarray*}

Case 3:  $x=(x_1,x_2)\in \sigma_{b}$ for $1\leq b\leq 2$.
$x_2=(2-b)(\gamma_0(x_1)+1)+2(b-1)$.
\begin{eqnarray*}\frac{d\varphi_1(x,t)}{dt}\mid_{t=0}&=&\eta(\varrho_2(x_1,b))(1,\,(2-b)\gamma_0'(x_1))\\
&=&\eta(\varrho_2(x_1,\frac{x_2-2\gamma_0(x_1)}{1-\gamma_0(x_1)}))(1,\,\frac{2-x_2}{1-\gamma_0(x_1)}\gamma_0'(x_1)).\end{eqnarray*}
Hence,
\begin{equation*}Z(x)=\begin{cases}\eta(\varrho_1(x_1))(1,\,\,a\gamma_0'(x_1)),&x\in U_1;\\
\eta(\varrho_2(x_1,x_2-\gamma_0(x_1)))(1,\,\gamma_0'(x_1)),&x\in V_1;\\
\eta(\varrho_2(x_1,\frac{x_2-2\gamma_0(x_1)}{1-\gamma_0(x_1)}))(1,\,\frac{2-x_2}{1-\gamma_0(x_1)}\gamma_0'(x_1)),&x\in
W_1;\\
 \eta(\varrho_2(x_1,x_2))(1,0)\,\,\, &x\in
\mathbb{R}^2\setminus (U_1\cup V_1\cup W_1).
\end{cases}\end{equation*}
It is verified that $Z(x)$ is $C^{2}$ with respect to $x\in U_1\cup
V_1$. Furthermore, $Z(x)=0$ for all $x=(x_1,x_2)\in U_1$ with
$-1\leq x_1\leq 1$.

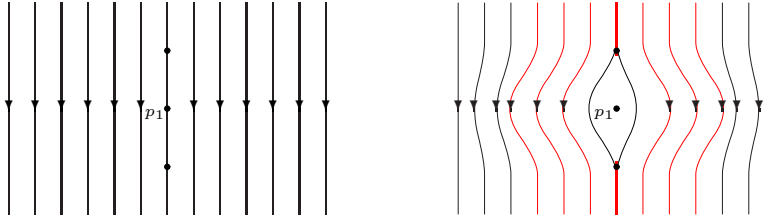
\begin{figure}[h]\label{iterate}
\begin{center}
\begin{picture}(180,120)(-30,-12)

\put(-50,0){\begin{picture}(100,100)

\put(0,-20){\line(0,1){80}}

\multiput(-10,-20)(-10,0){6}{\line(0,1){80}}
\multiput(10,-20)(10,0){6}{\line(0,1){80}} \put(0,20){\circle*{2}}

\put(0,42){\circle*{2}} \put(0,-2){\circle*{2}}

\multiput(-10,20)(-10,0){6}{\vector(0,-1){}}

\multiput(10,20)(10,0){6}{\vector(0,-1){}}

\put(-4.5,18){\makebox(0,0){\tiny $p_1$}}

\end{picture}
}

\put(120,0){\begin{picture}(100,100)

\spline(0,60)(0,50)(0,40)(-14,20)(0,0)(0,-10)(0,-20)
\spline(0,60)(0,50)(0,40)(10,20)(0,0)(0,-10)(0,-20)

{\color{red}\put(0,60){\line(0,-1){20}} \put(0,0){\line(0,-1){20}}
\multiput(0,0)(-10,0){3}{\spline(-10,60)(-10,50)(-10,40)(-24,20)(-10,0)(-10,-10)(-10,-20)}
\multiput(0,0)(10,0){3}{\spline(10,60)(10,50)(10,40)(24,20)(10,0)(10,-10)(10,-20)}
}

\spline(-40,60)(-40,50)(-40,40)(-47,20)(-40,0)(-40,-10)(-40,-20)
\spline(-50,60)(-50,50)(-50,40)(-55,20)(-50,0)(-50,-10)(-50,-20)
\spline(-60,60)(-60,50)(-60,40)(-60,20)(-60,0)(-60,-10)(-60,-20)

\spline(40,60)(40,50)(40,40)(47,20)(40,0)(40,-10)(40,-20)
\spline(50,60)(50,50)(50,40)(55,20)(50,0)(50,-10)(50,-20)
\spline(60,60)(60,50)(60,40)(60,20)(60,0)(60,-10)(60,-20)

\put(-20,20){\vector(0,-1){}}

\put(-30,20){\vector(0,-1){}}

\put(-40,20){\vector(0,-1){}}

\put(-45.5,20){\vector(0,-1){}}

\put(-54,20){\vector(0,-1){}}

\put(-60,20){\vector(0,-1){}}

\put(20,20){\vector(0,-1){}}

\put(30,20){\vector(0,-1){}}

\put(40,20){\vector(0,-1){}}

\put(45.5,20){\vector(0,-1){}}

\put(54,20){\vector(0,-1){}}

\put(60,20){\vector(0,-1){}}

\put(0,42){\circle*{2}} \put(0,-2){\circle*{2}}

 \put(0,20){\circle*{2}}
\put(-4.5,18){\makebox(0,0){\tiny $p_1$}}

\end{picture}
}

\end{picture}
\end{center}
\caption{\,Flows $\widehat{\phi}$ and $\varphi_1$ on $U_1\cup
V_1\cup W_1$}
\end{figure}


 In order to get a global flow on $\Omega$, we use the usual $(m+1)$-dimensional orthogonal group, denoted by $O(m)$, to rotate $Z$. For any
$A\in O(m)$, define
$$\widetilde{A}=\begin{pmatrix}1&0\\0&A\end{pmatrix}.$$
Let $\widetilde{U}_1=O(m)\varrho_2U_1$,
$\widetilde{V}_1=O(m)\varrho_2V_1$,
$\widetilde{W}_1=O(m)\varrho_2W_1$.
 Given $x=(x_1,\cdots,
x_{m+1})=\widetilde{A}\circ\varrho_2(x_1,\,\sqrt{\sum_{i=2}^{m+1}x_i^2})\in
B^{m+1}(0,4)$ for some $A\in O(m)$, denote the  vector field
$\widetilde{Z}$ as the rotation of $Z$ with the precise form
$$\widetilde{Z}(x)=\frac{d}{dt}\mid_{t=0}\widetilde{A}\circ\varrho_2\circ\varphi_1((x_1, \sqrt{\sum_{i=2}^{m+1}x_i^2}),t)=d(\widetilde{A}\circ\varrho_2)Z((x_1,\sqrt{\sum_{i=2}^{m+1}x_i^2})).$$
Together with the construction of $Z$, it follows that
\begin{equation*}\widetilde{Z}(x)=\begin{cases}
\eta(\varrho_1(x_1))(1,\,0,\cdots,0),\,\,\,(x_1,\sqrt{\sum_{i=2}^{m+1}x_i^2})\in U_1;\\
\eta(\varrho_2(x_1,\sqrt{\sum_{i=2}^{m+1}x_i^2}-\gamma(x_1)))(1,\,\gamma'(x_1)\frac{x_2}{\sqrt{\sum_{i=2}^{m+1}x_i^2}},\cdots,
\gamma'(x_1)\frac{x_{m+1}}{\sqrt{\sum_{i=2}^{m+1}x_i^2}}),\,\,\,\\\,\,\,\,\,(x_1,\sqrt{\sum_{i=2}^{m+1}x_i^2})\in
V_1.
\end{cases}
\end{equation*}
 Then $\widetilde{Z}(x)$ is $C^2$ on $\widetilde{U}_1\cup \widetilde{V}_1$ and the
corresponding flow $\phi_{\widetilde{Z}}$ is given by
$$\phi_{\widetilde{Z}}(x,t)=\widetilde{A}\circ\varrho_2\circ\varphi_1((x_1,\sqrt{\sum_{i=2}^{m+1}x_i^2}),t).$$
Now we define $\widetilde{\pi}: \widetilde{U}_1\cup
\widetilde{V}_1\cup \widetilde{W}_1\rightarrow \widetilde{U}_1\cup
\widetilde{V}_1\cup \widetilde{W}_1$ as follows
\begin{eqnarray*}
\widetilde{\pi}(x)&=&\varrho_1(x_1), \,\,\,\mbox{if}\,x\in \widetilde{A}\circ\varrho_2(\rho_a)\subset \widetilde{U}_1;\\
\widetilde{\pi}(x)&=&\widetilde{A}\circ\varrho_2((x_1,b)),
\,\,\,\mbox{if}\,x\in \widetilde{A}\circ\varrho_2(\sigma_b) \subset
\widetilde{V}_1\cup \widetilde{W}_1.
\end{eqnarray*}
Therefore,
$$\widetilde{\pi}\circ\phi_{\widetilde{Z}}=\varphi_0\circ \widetilde{\pi} \quad  \,\,\,\mbox{for}\,x\in \widetilde{U}_1\cup
\widetilde{V}_1\cup \widetilde{W}_1.$$

 (1) If $p\in \zeta^{-1}(\widetilde{U}_1\cup
\widetilde{V}_1\cup \widetilde{W}_1)$, define
$\widehat{\pi}=\zeta^{-1}\circ\widetilde{\pi}\circ\zeta$,
$\psi_1=\widehat{\psi}_1=\zeta^{-1}\circ\phi_{\widetilde{Z}}\circ\zeta$;

(2) If $p\in\Omega\setminus \zeta^{-1}(\widetilde{U}_1\cup
\widetilde{V}_1\cup \widetilde{W}_1)$,  define $\widehat{\pi}=\id$,
$\psi_1=\phi_{\widehat{\alpha}_1 X}$,
$\widehat{\psi}_1=\phi_{\alpha_1 X}$.

\begin{Prop}\label{zero-positive 1} $\psi_1$ and $\widehat{\psi}_1$
are  equivalent. Moreover, $$h(\psi_1)>K_1,\quad
 h(\widehat{\psi}_1)=0.$$
\end{Prop}
\begin{proof}
The equivalence of $\psi_1$ and $\widehat{\psi}_1$ are outputted by
their constructions. It is left to estimate the entropies of
$\psi_1$ and $\widehat{\psi}_1$. Observing that
$\phi_{\widehat{\alpha}_1 X}$ is actually a factor
 of $\psi_1$, that is, the following  graph is commutative
 \begin{eqnarray*}\Omega&\stackrel{\psi_1}{\longrightarrow}&\Omega\\
 \,\,\,\widehat{\pi}\big{\downarrow}&&\big{\downarrow}\widehat{\pi}\\
\Omega&\stackrel{\phi_{\widehat{\alpha}_1
X}}{\longrightarrow}&\Omega
 \end{eqnarray*}
$$\widehat{\pi}\circ\psi_1=\phi_{\widehat{\alpha}_1 X}\circ \widehat{\pi}.$$
So by  Theorem 7.2 of \cite{Walters} and Proposition \ref{positive
entropy},
$$h(\psi_1)\geq h(\phi_{\widehat{\alpha}_1 X})>K_1.$$

 Next we will
show that $h(\widehat{\psi}_1)=0$.  For every $p\in \Omega$ and
$t>0$, we need to estimate the proportion of its $t$-time orbit in
$\widetilde{U}$. Consider two transversal sections
$$H_1=\{x=(x_1,x_2,\cdots,x_{m+1})\in B^{m+1}(0,4)\mid x_1=3\}$$
$$H_2=\{x=(x_1,x_2,\cdots,x_{m+1})\in B^{m+1}(0,4)\mid x_1=-3\}.$$
Given $x\in H_1$, denote by $\tau(\phi,x)$ the first time $t>0$
satisfying that $\phi(x,t)\in H_2$. If $\phi(x,t)\notin H_2$ for all
$t>0$, we appoint $\tau(\phi, x)=\infty$. For $x\in H_1\setminus
\{x\mid \omega(x)\in F_0\}$  we claim that
$\phi_{\widetilde{Z}}(x,\tau(\phi_{\widetilde{Z}},x))
=\varphi_{0}(x,\tau(\varphi_{0}, x))$. To see why this is so, one
can use the fact that
$$\tau(\phi_{\widetilde{Z}},x)=\tau(\varphi_{0}, x)\quad
\mbox{and}\quad
\widetilde{\pi}\circ\phi_{\widetilde{Z}}=\varphi_0\circ
\widetilde{\pi}.$$ Exactly, $\widetilde{\pi}=\id$ for $x\in
\Omega\setminus (\widetilde{U}_1\cup \widetilde{V}_1\cup
\widetilde{W}_1)$. Thus
\begin{eqnarray*}\phi_{\widetilde{Z}}(x,\tau(\phi_{\widetilde{Z}},x))&=&\widetilde{\pi}\phi_{\widetilde{Z}}(x,\tau(\phi_{\widetilde{Z}},x))\\
&=&\varphi_{0}(\widetilde{\pi}(x),\tau(\varphi_{0},\widetilde{\pi}(x)))\\&=&\varphi_{0}(x,\tau(\varphi_{0},
x)).\end{eqnarray*} Recalling that for any open set $U_0\subset
\Omega\setminus (\widetilde{U}\cup \widetilde{V})$,
$$\lim_{t\rightarrow +\infty}\frac{J(t,p, \phi_{\alpha X}, U_0)}{t}=0,$$
we know
$$\lim_{t\rightarrow +\infty}\frac{J(t,p, \widehat{\psi}_1, U_0)}{t}=0,$$
which implies that  $\widehat{\psi}_1$ has no invariant measures on
$\Omega\setminus ( \{p_0\}\cup \widetilde{U})$. Therefore
$$h(\widehat{\psi}_1)=0.$$
\end{proof}

{\bf Step 4}\,\,\, Smoothness of the flows $\psi_1$ and
$\widehat{\psi}_1$ .\\

Define a $C^{\infty}$ function $v_0: \mathbb{R}^2\rightarrow
\mathbb{R}$ as follows
\begin{equation*}v_0(x_1,x_2)=\begin{cases}e^{\frac{1}{\gamma^2_0(x_1)-x_2}+\frac{1}{x_2-4}}\,\,&\mbox{for}\,\,\gamma^2_0(x_1)<x_2<4,\\
0\,\,&\mbox{otherwise},
\end{cases}\end{equation*}
which induces a new $C^{\infty}$ function $\widehat{v}_0:
\mathbb{R}^2\rightarrow \mathbb{R}$ given by
\begin{equation*}\widehat{v}_0(x_1,x_2)=\frac{\int_{x_2^2}^{4}
v_0(x_1,s)ds}{\int_{\gamma^2_0(x_1)}^{4} v_0(x_1,s)ds}.
\end{equation*}
We can  verify that $\widehat{v}_0$ satisfies\\

(1) $\widehat{v}_0(x_1,x_2)=0$ for $|x_2|\geq2$,

(2) $\widehat{v}_0(x_1,x_2)=1$ for $|x_2|\leq \gamma_0(x_1)$ and,

(3)
$\frac{\partial^{i+j}\widehat{v}_0}{\partial^ix_1\partial^jx_2}\mid_{x_2=2}
=\frac{\partial^{i+j}\widehat{v}_0}{\partial^ix_1\partial^jx_2}\mid_{x_2=\gamma_0(x_1)}=0$
for $i,j\geq 0$ and $i+j\geq 1$.\\

 Using  the function $\widehat{v}_0$, we
can define $C^{\infty}$ vector field $Z_1$ for  $(x_1,x_2)\in W_1$:
$$Z_1(x)=(\eta(x_1, x_2-\widehat{v}_0(x_1, x_2)\gamma_0(x_1))(1,\widehat{v}_0(x_1,x_2)\gamma_0'(x_1)).$$
Let $Z_1(x)=Z(x)$ for $x\in U_1\cup V_1$.
 Once more, we rotate  $Z_1$ by $O(m)$ to obtain a  vector field $\widetilde{Z}_1$ on $\widetilde{U}_1\cup \widetilde{V}_1\cup
 \widetilde{W}_1$. Then  $\widetilde{Z}_1(x)$  is $C^{\infty}$ on $\widetilde{V}_1\cup
 \widetilde{W}_1$. Moreover, $\widetilde{Z}_1=\widetilde{Z}$ is $C^2$
 on $\widetilde{U}_1\cup \widetilde{V}_1$.
 Consequently, $\widetilde{Z}_1$ is $C^2$ on $\widetilde{U}_1\cup \widetilde{V}_1\cup
 \widetilde{W}_1$.

 Let $\widehat{\psi}_2=\psi_2=\zeta^{-1}\circ\phi_{\widetilde{Z}_1}\circ \zeta$,
 when $p\in \widetilde{U}$; $\psi_2=\psi_1$, $\widehat{\psi}_2=\widehat{\psi}_1$, when $p\in \Omega\setminus
 \widetilde{U}$.
\begin{Prop}\label{zero-positive 2} $\psi_2$ and $\widehat{\psi}_2$
are
equivalent. In addition  $h(\widehat{\psi}_2)=0$ and there is a
constant $C_1>0$ independent of $K_1$ such that
 $$h(\psi_2)>C_1K_1.$$
\end{Prop}

\begin{proof}
 Noting the fact
$$\pi_i(\widetilde{Z}_1(x_1,x_2,\cdots,
x_{m+1}))=-\pi_i(\widetilde{Z}_1(-x_1,x_2,\cdots,x_{m+1}))$$ for
$2\leq i\leq m+1$, we deduce
\begin{eqnarray*}&&\pi_i(\phi_{\widetilde{Z}_1}(x,
\tau(x,\phi_{\widetilde{Z}_1})))\\[2mm]&=&x+\int_{0}^{\tau(x,\phi_{\widetilde{Z}_1})}\pi_i(\widetilde{Z}_1(\phi_{\widetilde{Z}_1}(x,s)))ds\\[2mm]
&=&x+\int_{0}^{\frac{\tau(x,\phi_{\widetilde{Z}_1})}{2}}\pi_i(\widetilde{Z}_1(\phi_{\widetilde{Z}_1}(x,s)))ds
+\int_{\frac{\tau(x,\phi_{\widetilde{Z}_1})}{2}}^{\tau(x,\phi_{\widetilde{Z}_1})}\pi_i(\widetilde{Z}_1(\phi_{\widetilde{Z}_1}(x,s)))ds\\[2mm]
&=&
x+\int_{0}^{\frac{\tau(x,\widetilde{\phi}_2)}{2}}\pi_i(\widetilde{Z}_1(\phi_{\widetilde{Z}_1}(x,s)))ds
-\int_{0}^{\frac{\tau(x,\widetilde{\phi}_2)}{2}}\pi_i(\widetilde{Z}_1(\phi_{\widetilde{Z}_1}(x,s)))ds\\[2mm]
&=&x\end{eqnarray*} for $2\leq i\leq m+1$, $x\in H_1\setminus \{p\in
\Omega\mid \omega(p)\in F_0\}$. Define $\widetilde{\pi}_1:
B^{m+1}(0,8)\rightarrow B^{m+1}(0,8)$ as follows
\begin{eqnarray*}&&\widetilde{\pi}_1(y)=(\pi_1(\phi_{\widetilde{Z}_1}( x, t)),\,
x_2,\,\cdots,\,x_{m+1}), \quad y=\phi_{\widetilde{Z}_1}( x,
t),\,\,x\in H_1,\,0\leq t< \tau(\varphi_2, x);\\
&&\widetilde{\pi}_1(y)=y, \quad \mbox{otherwise}.
\end{eqnarray*}  And further define $\widehat{\pi}_1: \Omega\rightarrow \Omega$ by
\begin{equation*}\widehat{\pi}_1=\begin{cases}\zeta^{-1}\circ \widetilde{\pi}_1\circ \zeta, &\mbox{if}\,\,p\in
\widetilde{U};\\
  \id, &\mbox{if}\,\,p\in \Omega\setminus \widetilde{U}.\end{cases}\end{equation*}  Then  $\psi_2$ and
$\widehat{\psi}_2$ are equivalent given by $\widehat{\pi}_1$.

Observing that there is no singularity in $W_1$, we  can choose
$0<C_1<1, C_2>1$ such that
$$C_1<\frac{\tau(\phi_{\widetilde{Z}_1},
x)}{\tau(\phi_{\widetilde{Z}}, x)}<C_2\quad \mbox{for}\,\,x\in
H_1\setminus \{p\mid \omega(p)\in A_0\}.$$

By Proposition \ref{zero-positive 1} and the variational principle,
there exists an ergodic measure $\mu_1$ of $\psi_1$ such that
$$h_{\mu_1}(\psi_1)>K_1.$$ Obviously, $$\supp(\mu_1)\cap
\zeta^{-1}(\widetilde{U}_1)=\emptyset.$$ Let $\nu_1=\mu_1\mid
\zeta^{-1}(H_1)$. Given a flow $\phi$ on $\Omega$, for any $p\in
\zeta^{-1}H_1$, denote by  $T(\phi, p)>0$ the first time of $p$
returning
 $\zeta^{-1}(H_1)$, and let the return map $$R(\phi, p)=\phi(p,T(\phi,p))\in \zeta^{-1}(H_1).$$
By Abarmov Theorem \cite{Abramov},
$$\frac{h_{\nu_1}(R(\psi_1))}{\int_{\zeta^{-1}(H_1)}T(\psi_1, p)d\nu_1}=h_{\mu_1}(\psi_1)>K_1.$$

  For $t>0$ large, define two sequences $\Gamma_i,
\Gamma_i'$ of sub-orbit of $\{\phi(p,s)\mid 0\leq s\leq t\}$ as
follows. We begin with $p$. Let $\Gamma_i$ be the sequence of
minimal intervals  whose left endpoint
 lies in $\zeta^{-1}(H_1)$ and right endpoint lies in $\zeta^{-1}(H_2)$. Let $\Gamma_i'$ be the sequence
 of minimal
intervals  whose left endpoint lies in $\zeta^{-1}(H_2)$ and right
endpoint lies in $\zeta^{-1}(H_1)$.

For each interval $\Gamma$, let $|\Gamma|$ denote the time of
sub-orbit $\Gamma$.

Since $\psi_2=\psi_1$ for $p\notin \widetilde{U}$, we have
\begin{eqnarray*}\frac{T(\psi_2,
p)}{T(\psi_1, p)}&=&\frac{\sum \Gamma_i(\psi_2)+\sum
\Gamma_i'(\psi_2)}{\sum \Gamma_i(\psi_1)+\sum
\Gamma_i'(\psi_1)}\\[2mm]&\leq&\frac{\sum \Gamma_i(\psi_2)+C_2\sum
\Gamma_i'(\psi_1)}{\sum \Gamma_i(\psi_1)+\sum
\Gamma_i'(\psi_1)}\\[2mm]&\leq&C_2,
\end{eqnarray*}and on the other hand,
\begin{eqnarray*}\frac{\sum \Gamma_i(\psi_2)+\sum
\Gamma_i'(\psi_2)}{\sum \Gamma_i(\psi_1)+\sum \Gamma_i'(\psi_1)}\geq
\frac{C_1\sum \Gamma_i(\psi_2)+\sum \Gamma_i'(\psi_2)}{\sum
\Gamma_i(\psi_1)+\sum \Gamma_i'(\psi_1)}\geq C_1.
\end{eqnarray*}
So, $$C_1<\frac{T(\psi_2, p)}{T(\psi_1, p)}<C_2 \quad
\mbox{for}\quad \,p\in \zeta^{-1} (H_1)\setminus \{p\in \Omega\mid
\omega(p)\in F_0\}.$$ We can define a measure
  $\mu_2$ by
$$\int_{\Omega}g d\mu_2:=\int_{\zeta^{-1}(H_1)}\int_0^{T(\psi_2,p)}g(\psi_2(p,t)) dtd\nu_1, \,\,\forall g \in C^0(\Omega).$$
Using Lemma \ref{suspend of measures}, $\mu_2$ is an ergodic
invariant measure of $\psi_2$. Furthermore, $R(\psi_1)=R(\psi_2)$
together with Abarmov Theorem \cite{Abramov}  yields that
\begin{eqnarray*}h_{\mu_2}(\psi_2)&=&\frac{h_{\nu_1}(R(\psi_2))}{\int_{\zeta^{-1}(H_1)}T(\psi_2,
p)d\nu_1}\\[2mm]
&=&\frac{h_{\nu_1}(R(\psi_1))}{\int_{\zeta^{-1}(H_1)}T(\psi_1,
p)d\nu_1}\frac{\int_{\zeta^{-1}(H_1)}T(\psi_1,
p)d\nu_1}{\int_{\zeta^{-1}(H_1)}T(\psi_2,
p)d\nu_1}\\[2mm]
&\geq& C_1h_{\mu_1}(\psi_1)\\[2mm]
&>&C_1K_1.\end{eqnarray*}
 Therefore
$$h(\psi_2)\geq h_{\mu_2}(\psi_{2})>C_1K_1.$$

Finally, since all invariant measures of $\widehat{\psi}_2$ are
supported on singularities  so $h(\widehat{\psi}_2)=0$.
\end{proof}

{\bf Step 5}\,\,\,\,Embed the two dimensional flows $\phi_{Z_1}$ and
$\phi_{Z_2}$ into $\Omega$.\\

At most taking a scallion of the coordinate $(\widetilde{U},\zeta)$,
we assume that $B^{m+1}(0,3)\subset \zeta(\widetilde{U}_1)$. In this
subsection, all modifications will be completed in $B^{m+1}(0,3)$.
Denote
\begin{eqnarray*}D_1&=&\{x\in
\mathbb{R}^{m+1}\mid\,x_1^2+x_2^2\leq1\,\,\mbox{and}\,\,\,\sum_{i=3}^{m+1}x_i^2=0\},\\
D_2&=&\{x\in \mathbb{R}^{m+1}\mid\,\,\,\,\sum_{i=1}^{m+1}x_i^2<2\}.
\end{eqnarray*}
We first choose $C^{\infty}$ smooth functions
$\varsigma,\widehat{z}_1, \widehat{z}_2, \widehat{\beta}_1,
\widehat{\beta}_2(x) : \widetilde{U}_1\rightarrow \mathbb{R}$ such
that
\begin{eqnarray*}\varsigma(x)\begin{cases}=0\,\,\,&x\in D_1,\\
>0\,\,\,&x\in D_2\setminus D_1,\\
=0\,\,\,&x\in \widetilde{U}_1\setminus
D_2;\end{cases}\end{eqnarray*}
\begin{eqnarray*}\widehat{z}_1(x)=\begin{cases}-x_2+\alpha(x_1^2+x_2^2)x_1\,\,\,&x\in D_1,\\
0\,\,\,&x\in \widetilde{U}_1\setminus D_2;\end{cases}\end{eqnarray*}
\begin{eqnarray*}\widehat{z}_2(x)=\begin{cases}-x_1+\alpha(x_1^2+x_2^2)x_2\,\,\,&x\in D_1,\\
0\,\,\,&x\in \widetilde{U}_1\setminus D_2;\end{cases}\end{eqnarray*}
\begin{eqnarray*}\widehat{\beta}_1(x)=\begin{cases}\beta_1(x)\,\,\,&x\in D_1,\\
\widehat{\beta}_1(x)>0\,\,\,&x\in D_2\setminus D_1\\
 1\,\,\,&x\in \widetilde{U}_1\setminus
D_2;\end{cases}\end{eqnarray*}
\begin{eqnarray*}\widehat{\beta}_2(x)=\begin{cases}\beta_2(x)\,\,\,&x\in D_1,\\
\widehat{\beta}_2(x)>0\,\,\,&x\in D_2\setminus D_1\\
=1\,\,\,&x\in \widetilde{U}_1\setminus
D_2.\end{cases}\end{eqnarray*}

Let
\begin{eqnarray*}\widehat{Z}_1(x)&=&\widehat{\beta}_1(x)(\widehat{z}_1(x),\widehat{z}_2(x),
\varsigma(x),
\cdots,\varsigma(x) ),\\
\widehat{Z}_2(x)&=&\widehat{\beta}_2(x)(\widehat{z}_1(x),\widehat{z}_2(x),
\varsigma(x),\cdots,\varsigma(x) ).\end{eqnarray*}

Noting that $\varsigma(x)>0$ for $x\in D_2\setminus D_1$, we know
that there is no nonwandering points in $D_2\setminus D_1$ for both
$\phi_{\widehat{Z}_1}$ and $\phi_{\widehat{Z}_2}$. Hence, all
invariant measures on $D_2 $  are supported on periodic orbits in
$D_1$, which implies no entropy production in $\widetilde{U}_1$.
Define
\begin{eqnarray*}\widehat{X}_1(p)=\begin{cases}(d\zeta^{-1})\widehat{Z}_1(\zeta(p))\,\,\,\,&p\in \zeta^{-1}(\widetilde{U}_1),\\
\frac{d\psi_2(p,t)}{dt}\mid_{t=0}\,\,&p\in \Omega\setminus
\zeta^{-1}(\widetilde{U}_1);\end{cases}\end{eqnarray*}
\begin{eqnarray*}\widehat{X}_2(p)=\begin{cases}(d\zeta^{-1})\widehat{Z}_2(\zeta(p))\,\,\,\,&p\in \zeta^{-1}(\widetilde{U}_1),\\
\frac{d\widehat{\psi}_2(p,t)}{dt}\mid_{t=0}\,\,&p\in \Omega\setminus
\zeta^{-1}(\widetilde{U}_1).\end{cases}\end{eqnarray*} Then
$\zeta\circ\phi_{\widehat{X}_1}\circ
\zeta^{-1}\mid_{D_1}=\phi_{Z_1}$,
$\zeta\circ\phi_{\widehat{X}_2}\circ
\zeta^{-1}\mid_{D_1}=\phi_{Z_2}$. By Theorem \ref{Main theorem1} it
holds that
$$EP(\phi_{\widehat{X}_1})=\infty
\,\,\,\mbox{and}\,\,\,EP(\phi_{\widehat{X}_2})=0.$$ Finally, let
$\psi=\phi_{\widehat{Z}_2}$, $\widehat{\psi}=\phi_{\widehat{Z}_1}$
and  take $K_1C_1>K$. We conclude that
$$EP(\psi)=0 \,\,\,\mbox{and}\,\,\,EP(\widehat{\psi})=\infty,$$
$$h(\psi)>K \,\,\,\mbox{and}\,\,\,h(\widehat{\psi})=0.$$

\section{Final Remarks  }
While our results give a very complete answer to the degeneration of
the growth of periodic orbits  for two-dimensional equivalent flows
 in the category of $C^{\infty}$ some interesting problems remain,
that we pose here

\begin{Que}Is it possible
to construct an analytic vector field or analytic map with
$EP=\infty$? Our method of proof clearly cannot be made analytic
since $\alpha_0$ is  flat at 0.  Noting that for any $k$-order
polynomial map $P_k$ on $\mathbb{R}^l$, any $n$ periodic point $x$
of $P_k$ satisfies
$$P_k^n(x)-x=0.$$ By the
Bezout theorem the number of isolated solutions is at most $k^{nl}$,
which implies $$EP(P_k)\leq
\limsup_{n\rightarrow+\infty}\frac{1}{n}\log(k^{nl})=l\log
k<\infty.$$  This fact make us tend to consider  the answer to be
negative.

\end{Que}
\begin{Que}Is the extreme $EP=0$ or $EP=\infty$ or the sign of $EP$ with finite value  preserved by
  orbit equivalent analytic flows? The flows $\psi$ and
$\widehat{\psi}$ are not analytic since $\omega_1$ and
$\widehat{\omega}_1$ are flat at $p_0$.
\end{Que}
\begin{Que}Is the extreme $EP=0$ or $EP=\infty$ or the sign of $EP$ with finite value  preserved by
   equivalent differential flows with only hyperbolic orbits? Recall
that a periodic orbit $\{\phi(x,t)\mid 0\leq t\leq T\}$ of period
$T$ is called hyperbolic  if the linearization $D\phi(\cdot\,,\,T)$
at $x$ has no  eigenvalue in the unity circle except the flow
direction. \end{Que}

We also have questions  concerning entropy $h$ and its relation to
$EP$ in smooth regularity.
\begin{Que}Is the value zero or the sign of entropy   preserved by
   equivalent analytic flows?
\end{Que}
Furthermore
\begin{Que}Are there two  equivalent $C^{\infty}$ or even analytic flows, one of
which has positive topological entropy  and zero exponential growth
rate of periodic orbits, in contrast,  the other has zero
topological entropy and super-exponential growth of periodic orbits?
In our constructions, $\eta$ is only $C^{2}$ and can't be improved
due to the appearance of square root when we use rotations of vector
fields on  $U_1\cup V_1$.
\end{Que}

\begin{Que}Besides entropy and the exponential growth of periodic orbits, are there other
objects invariant or decreasing for equivalent flows? Actually
physical measures \cite{Saghin-Sun-Vargas} and Lyapunov exponents
\cite{Gelfert-Motter} could decrease for equivalent flows.
\end{Que}


\begin{thebibliography}{10}



\bibitem{Abraham}R. Abraham, S. Smale, Nongenericity of $\Omega$-stability, Global
analysis I, {\it Proc. Symp. Pure Math. AMS} 14, 5--8, 1970.

\bibitem{Abramov-Rohlin}L. M. Abramov and V. A. Rohlin,  Entropy of a skew product of
mappings with invariant measure. {\it Vestnik Leningrad. Univ}. 17 ,
no. 7, 5--13, 1962.

\bibitem{Abramov}L. M. Abramov, On the entropy of a flow, {\it Dok. Akad. Nauk. SSSR}.
128, 873--875, 1959;  {\it Amer. Math. Soc, Trans}. 49, 167--170,
1966.

\bibitem{AM} M. Artin and B. Mazur, Periodic orbits, {\it Annals of Math}, 81,
82--99, 1965.


\bibitem{Birkhoff}G. D. Birkhoff, Nouvelles recherches sur les syst`emes
dynamiques. {\it Memoriae Pont. Acad. Sci. Novi Lyncaei } 1 (1935),
85--216 and Collected Math. Papers, vol. II, 530--659.

\bibitem{BDF} C. Bonatti, L.~J. D{\'\i}az, T. Fisher,
Super-exponential growth of the number of periodic orbits inside
homoclinic classes. {\it Discrete Contin. Dyn. Syst}. 20(3)
589--604, 2008.

\bibitem{Bonatti-Gan-Wen}C. Bonatti, S. Gan, L. Wen, On the existence of non-trivial
homoclinic classes. {\it Ergodic Theory \& Dynam. Systems} 27,
1473--1508, 2007.



\bibitem{Bowen3}R. Bowen,
 Topological entropy and Axiom A, {\it Global Analysis}
(Berkeley, CA, 1968, {\it Proc. Sympos. Pure Math}. 14, Amer. Math.
Soc., Providence, RI, 23--41, 1970.

\bibitem{Bowen}R. Bowen and P. Walters, Expansive one-parameter flows, {\it J.
Diff. Eq}., 12, 180--193, 1972.

\bibitem{Bowen2}R. Bowen,
Equilibrium states and the ergodic theory of Anosov diffeomorphisms,
Springer Lecture Notes in Math. 470, 1975.

\bibitem{Bowen1}R. Bowen, entropy and the fundamental group. The structure
of attractors in dynamical systems (Proc. Conf., North Dakota State
Univ., Fargo, ND, 1977), 21--29, Lectures Notes in Math. 668,
Springer-Verlag, New York, 1978.

\bibitem{BFF}M. Boyle, D. Fiebig, and U. Fiebig. Redidual entropy, conditional
entropy, and subshift covers. {\it Forum Math}., 14, 713--757, 2002.

\bibitem{Burguet2}D. Burguet, $C^2$ surface diffeomorphisms have symbolic
extensions,  to appear in {\it Invent. Math}.

\bibitem{Buzzi}J. Buzzi, Intrinsic ergodicity for smooth
interval maps. {\it Isreal J. Math}, 100, 125--161, 1997.

\bibitem{Crovisier}S. Crovisier, Birth of homoclinic intersections: a model for the
central dynamics of partially hyperbolic systems, {\it Annals of
Math }. 172. 3., 1641--1677, 2010.

\bibitem{DM}T. Downarowicz and A. Maass,
Smooth interval maps have symbolic extensions, {\it Invent. Math},
176(3), 617--636, 2009.

\bibitem{DN}T. Downarowicz, S. Newhouse, Symbolic extensions and
smooth dynamical systems, {\it Invent. math}, 160(3), 453--499,
2005.

\bibitem{Furstenberg} H. Furstenburg, Poincar{\'{e}} recurrence and number theory, {\it
Bulltin\, of A.M.S}. 5, 211--234, 1981.

\bibitem{Gelfert-Motter}K. Gelfert, A. E. Motter,  (Non)Invariance of dynamical quantities
for orbit equivalent flows, {\it Commun. Math. Phys}. 300, 411--433,
2010.


\bibitem{Guckenhei-Willians}J. Guckenheimer, R. F. Willians, Structural stability of Lorenz
attractors, {\it publ. Math. IHES} 50: 59--72, 1979.




\bibitem{Willians}J. Guckenheimer, R. F. Willians, The structure of the Lorenz
attractors, {\it publ. Math. IHES } 50: 73--99, 1979.




\bibitem{Liao-Sun-Tian}G. Liao, W. Sun, X. Tian, Metric entropy and the number of periodic
points, {\it Nonlinearity 23},  1547--1558, 2010.

\bibitem{Herman} M.R. Herman, Construction d¡¯un diffeomorphisme minimal
d¡¯entropie topologique non nulle, {\it Erg. Th \& Dyn. Systems}, 1,
65--76, 1981.

\bibitem{Kaloshin1}V. Y. Kaloshin, Generic diffeomorphisms with superexponential
growth of number of periodic orbits. {\it Comm. Math. Phys}. 211(1),
253--271, 2000.

\bibitem{Kaloshin2}V. Y. Kaloshin, An extension of the Artin-Mazur theorem, {\it Annals
of Math}, 150, 729--741, 1999.

\bibitem{Katok} A. Katok, Lyapunov exponents, entropy and periodic orbits for
diffeomorphisms, {\it Publ. IHES}, 51, 137--173, 1980.

\bibitem{STLiao} S.T. Liao, On the stability conjecture, {\it Chinese Ann. Math}. 1, no. 1, 9--30, 1980.

\bibitem{Maruyama} G. Maruyama, Theory of stationary processes and ergodic theory,
A lecture at the Symposium held at Kyoto Univ., 1965.

\bibitem{Newhouse}S. Newhouse, Nondensity of axiom $A$ on $\mathbb{S}^2$, Global analysis I,
{\it Proc. Symp. Pure Math. AMS} 14, 191--202, 1970.

\bibitem{New89}
S.~Newhouse.
\newblock Continuity properties of entropy.
\newblock {\em Annals of Math.}, 129:215--235, 1990.


\bibitem{Ohno} T. Ohno, A weak equivalence and topological entropy, {\it Publ. RIMS,
Kyoto Univ}. 16, 289--298, 1980.


\bibitem{Palis}J. Palis, Open questions leading to a global perspective in
dynamics,  {\it Nonlinearity} 21,  T37--T43, 2008.


\bibitem{Po1}H. Poincar{\'e}, Sur le probl{\'e}me des trois corps et les
equations de la dynamique. {\it Acta Math}. 13, 1--270, 1890.


\bibitem{Pujals-Sambarino}E. Pujals, M. Sambarino, Homoclinic tangencies and hyperbolicity
for surface diffomorphisms, {\it Ann. of Math}. 151, no. 3,
961-1023, 2000.

\bibitem{SunVar} W. Sun, E. Vargas, Entropy of flows, revisited, {\it Bol. Soc.
Brasil. Mat}. 30, 315--333, 1999.

\bibitem{Sun} W. Sun, Entropy of orthonormal $n$-frame flows, {\it Nonlinearity}, 14, no. 4, 829--842, 2001.

\bibitem{Saghin-Sun-Vargas}R. Saghin, W. Sun, E. Vargas, On Dirac physical measures for transitive
flows, {\it Communications in Mathematical Physics}: 298(3),
741--756, 2010.

\bibitem{Smale}S. Smale, Diffeomorphisms with many periodic points. {\it Differential
and combinatorial topology}, Princeton Univ. Press, 63--80, 1965.

\bibitem{SunYoungZhou} W. Sun, T. Young, Y. Zhou, Topological entropies
of equivalent smooth flows. {\it Trans. Amer. Math. Soc}. 361, no.
6, 3071--3082, 2009.

\bibitem{SunZhang} W. Sun, C. Zhang,
Extreme growth rates of periodic orbits in equivalent flows, to
appear in {\it Proc. Amer. Math. Soc}.


\bibitem{Sun-Zhang-Zhou} W. Sun, C. Zhang, Y. Zhou,
 Extreme entropy versus extreme growth rates of periodic orbits in equivalent
flows, Preprint, 2010


\bibitem{Thomas1} R. Thomas, Topological entropy of fixed-point free flows, {\it Trans.
Amer. Math. Soc}., 319, 601--618, 1990.


\bibitem{Thomas2} R. Thomas, Entropy of expansive flows, {\it Erg. Th \& Dyn. Systems}, 7, 611--625, 1987.

\bibitem{Totoki} H. Totoki, Time changes of flows, {\it Mem. Fac. Sci. Kyushu
Univ}. Ser. A, 20, 27--55, 1966.

\bibitem{LSYoung}L. S. Young, Entropy of continuous flows on compact 2-manifolds,  {\it Topology} 16 (4), 469--471, 1977.

\bibitem{Walters} P. Walters, An introduction to ergodic theory, Springer-Verlag,
1982.

\end{thebibliography}
\end{document}